\documentclass{article}
\usepackage{geometry}                
\usepackage{setspace}
\usepackage{graphicx}
\usepackage{amssymb}
\usepackage{epstopdf}
\usepackage{enumerate}
\usepackage[fleqn]{amsmath}
\usepackage{xfrac}
\usepackage{titlesec}
\usepackage[titletoc,toc,title]{appendix}
\usepackage{float}
\usepackage{bbm}
\usepackage[]{algorithm2e}
\usepackage{graphicx}
\usepackage[font={small,it}]{subcaption}
\usepackage[font={small,it}]{caption}
\usepackage[justification=centering]{caption}
\usepackage{mwe}
\usepackage{multirow}
\usepackage[bottom]{footmisc}
\usepackage{booktabs}
\usepackage{amsthm}
\usepackage{bbm}
\restylefloat{table}
\usepackage{rotating}
\titleformat*{\section}{\large\bfseries}
\titleformat*{\subsection}{\normalsize\bfseries}
\newtheorem{thm}{Theorem}
\newtheorem*{thm*}{Theorem}
\newtheorem{proposition}{Proposition}
\usepackage{natbib}
\usepackage{abstract,lipsum}
\title{\bf Limit theorems for the Multiplicative Binomial Distribution (MBD)}

\author{F. Fortunato \\ Department of Statistical Sciences, University of Bologna, Italy }
\date{\today}
\begin{document}

\maketitle

\renewenvironment{abstract}
{\begin{quote}
\noindent \rule{\linewidth}{.5pt}\par{\bfseries \abstractname.}}
{\medskip \rule{\linewidth}{.5pt}
\end{quote}
}

\begin{abstract}
The sum of $n$ {non-independent} Bernoulli random variables could be modeled in several different ways. One of these is the Multiplicative Binomial Distribution (MBD), introduced by Altham (1978) and revised by Lovison (1998). In this work, we focus on the distribution asymptotic behavior as its parameters diverge. In addition, we derive a specific property describing the relationship between the joint probability of success of $n$ binary-dependent responses and the individual Bernoulli one; particularly, we prove that it depends on both the sign and the strength of the association between the random variables.\\
\end{abstract}

\section{Introduction}
Let $Z_i$ be a binary response measuring, for an event of interest, its {\em presence} (`1', success) or {\em absence} (`0', failure) and let $Y_n = \sum_{i=1}^{n} Z_i$ be the number of successes in a sequence of $n$ trials.  \\
In the case of independent trials, it is well known that $Y_n$ follows a Binomial distribution, $Y_n \sim Bin (\pi)$, where $\pi = P(Z_i =1)$ is the fixed trial probability of success. \\
In the case of non-independent trials, the Binomial extension is not unique. In recent years, several approaches have been discussed in order to accomodate the association among the $Z_i$s, as dependent or correlated binary data are becoming more and more common in many application areas \citep{zhao1990correlated} (e.g. studies of disease occurrence among family members, analyses with repeated measurements on study subjects, longitudinal series, researches involving group randomization, ensemble classification, \dots). The literature about the sums of non-independent Bernoulli random variables shows different possible strategies for dealing with the ``intra-units" association: Skellam \citep{skellam1948probability} proposed to model the $\pi$ parameter of the Binomial Distribution with a Beta$(\alpha,\beta)$ model; Altham \citep{altham1978two} discussed the possibility of extending the Binomial model in two different directions, the Additive Binomial Distribution and the Multiplicative Binomial Distribution, respectively characterized by an `additive' and a `multiplicative' definition of the interaction among units; Diniz {\em et al.} \citep{diniz2010bayesian} applied a Bayesian approach to the Correlated Binomial model introduced by Luce{\~n}o \citep{luceno1995family}; Kadane \citep{kadane2016sums} derived the Conway-Maxwell-Binomial Distribution so as to model both positive and negative dependence among the Bernoulli summands. \\

In this work, with the aim of modeling the non-independence, we focused our attention on the {\em Multiplicative Binomial Distribution} (MBD) introduced in \citep{altham1978two} and based on the original Cox's log-linear representation (1972), by studying its asymptotic behavior. Specifically, we refer to the revised version of that distribution, named {\em Lovison's Multiplicative Binomial Distribution} (LMBD), introduced by Lovison \citep{lovison1998alternative} and characterized by a more intuitive interpretation of the distribution parameters. Such a distribution is a member of the exponential family and therefore it has sufficient statistics and a family of proper conjugate distributions. \\
Under the assumption of exchangeable units, the (L)MBD takes the form:

\begin{equation*}
P(Y_n=y) = \frac{\binom{n}{y} \psi^y (1-\psi)^{n-y} \omega^{(n-y)y}}{\sum_{i=0}^{n} \binom{n}{i} \psi^i (1-\psi)^{n-i} \omega^{(n-i)i}}.
\label{eq:LMBD}
\end{equation*}

Here:
\begin{itemize}
\item $\psi$, $0 \le \psi =\pi/\tau_1 \le 1$ is the {\bf independence marginal probability} parameter (i.e. in the case of {\em independent} trials $\psi=\pi$), where
\[
\tau_r(\psi,\omega)=\frac{K_{n-r}(\psi,\omega)}{K_{n}(\psi,\omega)} \qquad r = 1, \dots, n
\]
and 
\[
K_{n-a}(\psi,\omega) = \sum_{i=0}^{n-a} \binom{n-a}{i} \psi^i (1-\psi)^{n-a-i} \omega^{(n-a-i)(i+a)};
\]

$\psi$ could be less than or larger than $\pi$, depending on $\tau_1(\psi, \omega)$.
\item $\omega > 0$ is the {\bf intra-units association} parameter which governs the dependence between the trials: $\omega<1$ describes positively associated variables, $\omega>1$ a negative global relationship and $\omega = 1$ independent trials. \\
This measure is inversely related to the conditional cross-product ratio (CPR) as
\begin{equation*}
\label{omega}
\begin{aligned}
&\omega_{i,j} = \frac{1}{\sqrt{CPR_{i,j}|rest}}\mbox{,} \\
&CPR_{i,j}|rest=\frac{P(Z_i=1, Z_j=1)P(Z_i=0,Z_j=0)}{P(Z_i=1, Z_j=0)P(Z_i=0, Z_j=1)}, \qquad i,j = 1, \cdots, n, \quad i \ne j.
\end{aligned}
\end{equation*}
\end{itemize}

In \citep{lovison1998alternative}, Lovison also derived the first two central moments of the (L)MBD in a form that facilitates their comparison to the binomial ones:
\begin{equation*}
\begin{aligned}
E[Y_n]&=n\psi\tau_1 \\
V[Y_n]&=n\psi\eta 
\end{aligned}
\end{equation*}
where $\eta= \tau_1-\psi(n\tau_1^2-(n-1)\tau_2)$.
\bigskip

In section 2, the limits of the (L)MBD are investigated. In particular, Theorem~\ref{theorem1} proves the convergence of the (L)MBD to the Dirac-Delta, $\delta$, when both its parameters $\omega$ and $\psi$ diverge; then, Proposition~\ref{prop1} shows the asymptotical behavior of the distribution, as $n$ increases; lastly, Theorem~\ref{theorem2} describes the relationship between the parameters of the (L)MBD $\omega$ and $\psi$ and the probability of success of a single trial, $\pi$. Each statement is followed by its mathematical proof.

\bigskip
\section{Limit theorems of (L)MBD}
\begin{thm}
\label{theorem1}
Let $Y_n \sim (L)MBD(\psi, \omega)$, $n$ be the number of trials and $k$ a positive integer: \\
\begin{itemize}
\item $\forall n$:
\begin{equation*}
Y_n \xrightarrow[\omega \rightarrow 0^+]{d} \begin{cases} \delta\textnormal{(0)} & \text{if} \hspace{0.2cm} \psi \rightarrow 0 \\
		\delta\textnormal{(n)} & \text{if} \hspace{0.2cm} \psi \rightarrow 1
		\end{cases}
\end{equation*}

\item $\forall n=2k$:
\begin{equation*}
Y_n \xrightarrow[\omega \rightarrow +\infty]{d} \delta\textnormal{$\left(\frac{\rm{n}}{2}\right)$}
\end{equation*}

\item $\forall n=2k+1$:
\begin{equation*}
Y_n \xrightarrow[\omega \rightarrow +\infty]{d} \begin{cases}\delta\textnormal{$\left(\frac{\rm{n-1}}{2}\right)$} & \text{if} \hspace{0.2cm} \psi \rightarrow 0 \\
		\delta\textnormal{$\left(\frac{\rm{n-1}}{2}+1\right)$} & \text{if} \hspace{0.2cm} \psi \rightarrow 1     \\
		\end{cases}
\end{equation*}

\end{itemize}
\end{thm}

\begin{proof}
{\bf Case 1} - Positive association ($\omega<1$), $\forall n$: \\
\begin{equation*}
\begin{aligned}
\lim_{\omega \to 0^+} \tau_j
&= \lim_{\omega \to 0^+} \frac{\sum_{i=0}^{n-j} \binom{n-j}{i} \psi^i (1-\psi)^{n-j-i} \omega^{(n-j-i)(i+j)}}{\sum_{i=0}^{n} \binom{n}{i} \psi^i (1-\psi)^{n-i} \omega^{(n-i)i}} \\
&=\frac{\lim_{\omega \to 0^+}\sum_{i=0}^{n-j} \binom{n-j}{i} \psi^i (1-\psi)^{n-j-i} \omega^{(n-j-i)(i+j)}}{\lim_{\omega \to 0^+}\sum_{i=0}^{n} \binom{n}{i} \psi^i (1-\psi)^{n-i} \omega^{(n-i)i}} \\
&= \frac{\psi^{n-j}}{\psi^n+(1-\psi)^n} , \qquad j \le {n}, \qquad \tau_j = \mathcal O(\omega^{{n}-1}).\\
\end{aligned}
\end{equation*}

Therefore,

\begin{equation*}
\begin{aligned}
\lim_{\omega \to 0^+} E[Y_n] &= \lim_{\omega \to 0^+} n\psi\tau_1 = n\psi\lim_{\omega \to 0^+} \tau_1 = \frac{n\psi^{n}}{\psi^n+(1-\psi)^n} \\
\lim_{\omega \to 0^+} V[Y_n] &= \lim_{\omega \to 0^+} n\psi\eta = \lim_{\omega \to 0^+} n\psi [\tau_1-\psi(n\tau_1^2-(n-1)\tau_2)] = n\psi [\lim_{\omega \to 0^+} \tau_1-\psi(n\lim_{\omega \to 0^+}\tau_1^2-(n-1)\lim_{\omega \to 0^+}\tau_2)]  \\
& = n\psi \left[\frac{\psi^{n-1}}{\psi^n+(1-\psi)^n} - \psi n \left(\frac{\psi^{n-1}}{\psi^n+(1-\psi)^n}\right)^2 + \psi (n-1) \frac{\psi^{n-2}}{\psi^n+(1-\psi)^n}\right]\\
&=n\psi \left[\frac{\psi^{n-1}}{\psi^n+(1-\psi)^n} - \psi n \left(\frac{\psi^{n-1}}{\psi^n+(1-\psi)^n}\right)^2 + \psi n \frac{\psi^{n-2}}{\psi^n+(1-\psi)^n} - \psi \frac{\psi^{n-2}}{\psi^n+(1-\psi)^n}\right] \\
&=n\psi \left[\frac{\psi^{n-1}}{\psi^n+(1-\psi)^n} - \psi n \left(\frac{\psi^{n-1}}{\psi^n+(1-\psi)^n}\right)^2 + \psi n \frac{\psi^{n-2}}{\psi^n+(1-\psi)^n} - \frac{\psi^{n-1}}{\psi^n+(1-\psi)^n}\right] \\
&= n\psi \left[\frac{n \psi^{n-1}}{\psi^n+(1-\psi)^n} - \psi n \frac{\psi^{2n-2}}{[\psi^n+(1-\psi)^n]^2}\right] \\
&= n\psi \left[\frac{n \psi^{n-1}}{\psi^n+(1-\psi)^n} - \frac{n \psi^{2n-1}}{[\psi^n+(1-\psi)^n]^2} \right]\\
&= \frac{n^2 \psi^{n}}{\psi^n+(1-\psi)^n} - \frac{n^2 \psi^{2n}}{[\psi^n+(1-\psi)^n]^2}.
\end{aligned}
\end{equation*}

It follows that:
{\footnotesize
\begin{equation*}
\begin{aligned}
&\lim_{\psi \to 0^+} \left(\lim_ {\omega \to 0^+} E[Y_n]\right) = \lim_{\psi \to 0^+}  \left[\frac{{n}\psi^{{n}}}{\psi^{n}+(1-\psi)^{n}} \right]=0, \quad \mbox{where} \quad \frac{{n}\psi^{{n}}}{\psi^{n}+(1-\psi)^{n}} = \mathcal O(\psi^{{n}-1})\\
&\lim_{\psi \to 0^+} \left(\lim_ {\omega \to 0^+} V[Y_n]\right) = \lim_{\psi \to 0^+} \left[\frac{{n}^2 \psi^{{n}}}{\psi^{n}+(1-\psi)^{n}} - \frac{{n}^2 \psi^{2{n}}}{[\psi^{n}+(1-\psi)^{n}]^2}\right]=0, \quad \mbox{where} \quad \frac{{n}^2 \psi^{{n}}}{\psi^{n}+(1-\psi)^{n}} - \frac{{n}^2 \psi^{2{n}}}{[\psi^{n}+(1-\psi)^{n}]^2} = \mathcal O(\psi^{{n}-1}) \\
&\lim_{\psi \to 1^-} \left(\lim_ {\omega \to 0^+} E[Y_n]\right) = \lim_{\psi \to 1^-}  \left[\frac{{n}\psi^{{n}}}{\psi^{n}+(1-\psi)^{n}} \right]={n},  \quad \mbox{where} \quad  \frac{{n}\psi^{{n}}}{\psi^{n}+(1-\psi)^{n}} = \mathcal O(\psi-1)\\
&\lim_{\psi \to 1^-} \left(\lim_ {\omega \to 0^+} V[Y_n]\right) = \lim_{\psi \to 1^-} \left[\frac{{n}^2 \psi^{{n}}}{\psi^{n}+(1-\psi)^{n}} - \frac{{n}^2 \psi^{2{n}}}{[\psi^{n}+(1-\psi)^{n}]^2}\right]=0,  \quad \mbox{where} \quad \frac{{n}^2 \psi^{{n}}}{\psi^{n}+(1-\psi)^{n}} - \frac{{n}^2 \psi^{2{n}}}{[\psi^{n}+(1-\psi)^{n}]^2} = \mathcal O(\psi^{{n}}) 
\end{aligned}
\end{equation*}}

\bigskip
{\bf Case 2.1} - Negative association ($\omega>1$) with an even number of trials ($n=2k$): \\
\begin{equation*}
\begin{aligned}
\lim_{\omega \to +\infty} \tau_j
&= \lim_{\omega \to +\infty} \frac{\sum_{i=0}^{n-j} \binom{n-j}{i} \psi^i (1-\psi)^{n-j-i} \omega^{(n-j-i)(i+j)}}{\sum_{i=0}^{n} \binom{n}{i} \psi^i (1-\psi)^{n-i} \omega^{(n-i)i}} \\
&=\frac{\lim_{\omega \to +\infty}\sum_{i=0}^{n-j} \binom{n-j}{i} \psi^i (1-\psi)^{n-j-i} \omega^{(n-j-i)(i+j)}}{\lim_{\omega \to +\infty}\sum_{i=0}^{n} \binom{n}{i} \psi^i (1-\psi)^{n-i} \omega^{(n-i)i}} \\
&= \frac{1}{\psi^j}\frac{\binom{n-j}{\frac{n}{2}-j}}{\binom{n}{\frac{n}{2}}}\qquad j \le \frac{{n}}{2}, \qquad \tau_j=\mathcal O\left(\frac{1}{\omega}\right).\\
\end{aligned}
\end{equation*}

Therefore,
\begin{equation*}
\begin{aligned}
\lim_{\omega \to +\infty} E[Y_n] &= \lim_{\omega \to +\infty} n\psi\tau_1 = n\psi\lim_{\omega \to +\infty} \tau_1 = n\psi\frac{1}{2\psi} = \frac{n}{2} \\
\lim_{\omega \to +\infty} V[Y_n] &= \lim_{\omega \to +\infty} n\psi\eta = \lim_{\omega \to +\infty} n\psi [\tau_1-\psi(n\tau_1^2-(n-1)\tau_2)] = n\psi [\lim_{\omega \to +\infty} \tau_1-\psi(n\lim_{\omega \to +\infty}\tau_1^2-(n-1)\lim_{\omega \to +\infty}\tau_2)]  \\
&= n\psi \left[\frac{1}{2\psi}  - \psi n \left(\frac{1}{2\psi} \right)^2 +\psi(n-1)\frac{n-2}{4(n-1)\psi^2}\right] \\
&=n\psi \left[\frac{1}{2\psi}  - \psi n \frac{1}{4\psi^2}+\frac{n-2}{4\psi}\right]\\
&=n\psi \left[\frac{1}{2\psi}  - \frac{n}{4\psi}+\frac{n-2}{4\psi}\right]\\
&=n\psi \left[\frac{2-n+n-2}{4\psi}\right] = 0.
\end{aligned}
\end{equation*}

\bigskip
{\bf Case 2.2}: Negative association ($\omega>1$) with an odd number of trials ($n=2k+1$): \\
\begin{equation*}
\begin{aligned}
\lim_{\omega \to +\infty} \tau_j
&= \lim_{\omega \to +\infty} \frac{\sum_{i=0}^{n-j} \binom{n-j}{i} \psi^i (1-\psi)^{n-j-i} \omega^{(n-j-i)(i+j)}}{\sum_{i=0}^{n} \binom{n}{i} \psi^i (1-\psi)^{n-i} \omega^{(n-i)i}} \\
&=\frac{\lim_{\omega \to +\infty}\sum_{i=0}^{n-j} \binom{n-j}{i} \psi^i (1-\psi)^{n-j-i} \omega^{(n-j-i)(i+j)}}{\lim_{\omega \to +\infty}\sum_{i=0}^{n} \binom{n}{i} \psi^i (1-\psi)^{n-i} \omega^{(n-i)i}} \\
&= \frac{1}{\psi^j} \frac{R^{-1}+\psi}{\frac{\binom{n}{\frac{n-1}{2}}}{\binom{n-j}{\frac{n-1}{2}-j}}R^{-1}} , \qquad j \le \frac{{n}-1}{2}, \qquad \tau_j = \mathcal O\left(\frac{1}{\omega^2}\right).
\end{aligned}
\end{equation*}
where $R=\frac{\binom{n-j}{\frac{n-1}{2}-j+1}}{\binom{n-j}{\frac{n-1}{2}-j}}-1$.

Therefore,
\begin{equation*}
\begin{aligned}
\lim_{\omega \to +\infty} E[Y_n] &= \lim_{\omega \to +\infty} n\psi\tau_1 = n\psi\lim_{\omega \to +\infty} \tau_1 = n\psi\frac{\frac{n-1}{2}+\psi}{n \psi} = \frac{n-1}{2}+\psi\\
\lim_{\omega \to +\infty} V[Y_n] &= \lim_{\omega \to +\infty} n\psi\eta = \lim_{\omega \to +\infty} n\psi [\tau_1-\psi(n\tau_1^2-(n-1)\tau_2)]\\
&= n\psi [\lim_{\omega \to +\infty} \tau_1-\psi(n\lim_{\omega \to +\infty}\tau_1^2-(n-1)\lim_{\omega \to +\infty}\tau_2)]  \\
&= n\psi \left[\frac{\frac{n-1}{2}+\psi}{n\psi}-n\psi\left(\frac{\frac{n-1}{2}+\psi}{n\psi}\right)^2 + (n-1)\psi\left( \frac{\frac{n-3}{4}+\psi}{n\psi^2} \right)\right] \\
&= n\psi \left[\frac{\frac{n-1}{2}+\psi}{n\psi}-  \frac{(\frac{n-1}{2}+\psi)^2}{n\psi}+ \frac{(n-1)\left(\frac{n-3}{4}+\psi\right)}{n\psi}\right]\\
&= n\psi \left[ \frac{\frac{n-1}{2}+\psi-\frac{(n-1)^2}{4}-\psi^2-(n-1)\psi+\frac{(n-1)(n-3)}{4}+(n-1)\psi}{n\psi} \right]\\
&= n\psi \left[ \frac{2n-2+4\psi-n^2-1+2n-4\psi^2+n^2-3n-n+3}{4n\psi}\right]\\
&= n\psi \left[\frac{4\psi-4\psi^2}{4n\psi}\right] = \psi(1-\psi).
\end{aligned}
\end{equation*}

It follows that:
{
\begin{equation*}
\begin{aligned}
&\lim_{\psi \to 0^+} \left(\lim_ {\omega \to +\infty} E[Y_n]\right) = \lim_{\psi \to 0^+} \left[ \frac{{n}-1}{2}+\psi\right] = \frac{{n}-1}{2},  \quad \mbox{where} \quad \frac{{n}-1}{2}+\psi = \mathcal O(\psi) \\
&\lim_{\psi \to 0^+} \left(\lim_ {\omega \to +\infty} V[Y_n]\right) = \lim_{\psi \to 0^+} \left[ \psi(1-\psi)\right] = 0,  \quad \mbox{where} \quad \psi(1-\psi)=\mathcal O(\psi^2)\\
&\lim_{\psi \to 1^-} \left(\lim_ {\omega \to +\infty} E[Y_n]\right) = \lim_{\psi \to 1^-} \left[ \frac{{n}-1}{2}+\psi \right] = \frac{{n}-1}{2},  \quad \mbox{where} \quad  \frac{{n}-1}{2}+\psi =\mathcal O(\psi-1) \\
&\lim_{\psi \to 1^-} \left(\lim_ {\omega \to +\infty} V[Y_n]\right) = \lim_{\psi \to 1^-} \left[ \psi(1-\psi) \right] = 0,  \quad \mbox{where} \quad  \psi(1-\psi)=\mathcal O(\psi-1)
\end{aligned}
\end{equation*}
}

\vfill
Combining these results, it is straightforward to notice that, in all the cases where the limit of the variance is equal to 0, the random variable $Y_n$ degenerates to the limit of its expectation, $L_E$, with probability 1. Formally,
\[
P(Y_n=y) = \begin{cases} 1 & \mbox{if} \hspace{0.2cm} y = L_E\\ 0 & \text{otherwise} \end{cases} \hspace{0.5cm}\Longrightarrow \hspace{0.5cm} Y_n \xrightarrow[\begin{subarray} {}   \omega \rightarrow 0^+ \ \vee \ \omega \rightarrow +\infty \\ \psi \rightarrow 0^+ \ \vee \ \psi \rightarrow 1^-\end{subarray}]{d} \delta[L_E],
\]

where $\delta$ is the Dirac-Delta function $\delta_{x_0}[\phi] = \phi(x_0)$.
\end{proof}

\begin{proposition}
\label{prop1}
Let $Y_n \sim (L)MBD(\psi, \omega)$, $n$ be the number of trials,
\begin{equation*}
Z=\frac{Y_n-n\psi \tau_1}{\sqrt{n\psi\eta}} \xrightarrow[n \rightarrow +\infty]{d} \mathcal{N} (0, 1),
\end{equation*}
where $\mathcal{N}$ is the Gaussian distribution.
\end{proposition}

\begin{proof}
Because of the symmetry of the joint distribution of $(Z_1, ..., Z_n)$, it is always possible to write:
\[E[Y_n] = n E[Z_1]\] and
\[V[Y_n] = nV[Z_1] +  n(n-1) Cov[Z_1,Z_2].\]
Therefore, the Central Limit Theorem for dependent random variables can be applied to $Y_n$, provided that the overall mean and variance behave `sensibly'. \\
Specifically, we refer to the central limit theorem for dependent classes of random variables derived by Kaminski in \citep{kaminski2007central}:

\begin{thm*}
Let $\left\{X_i\right\}_{i \ge 1}$ be a sequence of identically distributed random variables such that $E|X_1|^{2+\epsilon}<+\infty$ for some $\epsilon >0$. Let $V[X_1]=\sigma^2$ and $\epsilon_1$ be a positive number such that $\epsilon_1<\frac{\epsilon}{2(1+\epsilon)}$. Denote by $S=\sum_{i=1}^n X_i$ the partial sum. Suppose that for sufficiently large $k$, the inequality
\begin{equation}
\label{cond1}
\sup\left\{|P\left(\bigcap_{i=1}^{j} \{X_{v_i} \le x_{v_i}\}\right)  - \prod_{i=1}^{j}P (X_{v_i} \le x_{v_i}): (x_{v_1},\dots,x_{v_j}) \in \mathbb{R}^j  |  \right\} \le (1-k^{-\epsilon_1})^{k-k^{\epsilon_1}-j}
\end{equation}
holds, where ${v_1},\dots,v_j$ is any choice of indices such that $k^{\epsilon_1}<v_1<\dots<v_j\le k$. Then: 
\[
\frac{S-E[S]}{\sigma\sqrt{n}} \xrightarrow[]{d} N(0,1) \qquad \hbox{as } n\to \infty.
\]
\end{thm*} 
It is important to underline that the left-hand side of condition~\ref{cond1} is only on the tail $X_{k^{\epsilon_1}},X_{k^{\epsilon_1}+1},\dots,X_{k}$ and it reflects the degree of dependence among $X_{v_1},\dots,X_{v_j}$ (i.e. if $X_{v_1},\dots,X_{v_j}$ are independent, the left-hand side of inequality~\ref{cond1} is 0, otherwise it is a real positive number).
Then, it is easy to see that, for fixed $k$, the right-hand side of~\ref{cond1} tends to become larger as $j$ increases. \\

In our case, a different number ${n}$ of Bernoulli variables $Z_1, ..., Z_n$ is required so as to satisfy inequality~\ref{cond1}, depending on their average degree of dependence, ${\omega}$. Condition~\ref{cond1} surely holds for any ${\omega}$ if ${n}\to \infty$. 
\end{proof}

%

\bigskip

\begin{thm}
\label{theorem2}
Let $Y_n \sim (L)MBD(\psi, \omega)$, $n$ be the number of trials and $\psi=\pi / \tau_1 \ge \frac{1}{2}$:
\begin{equation*}
\omega>1 \Rightarrow \psi > \pi
\end{equation*}
\end{thm}

\begin{proof}
\begin{equation*}
\pi(\psi,\omega)=\psi \tau_1
\end{equation*}
Now,
\begin{equation}
\label{tau}
\begin{aligned}
\tau_1 &= \frac{\sum_{i=0}^{n-1} \binom{n-1}{i} \psi^i (1-\psi)^{n-1-i} \omega^{(n-1-i)(i+1)}}{\sum_{i=0}^{m} \binom{m}{i} \psi^i (1-\psi)^{m-i} \omega^{(n-i)i}} \le 1  \iff \\
&D_n = \sum_{i=0}^{n-1} \binom{n-1}{i} \psi^i (1-\psi)^{n-1-i} \omega^{(n-1-i)(i+1)} - \sum_{i=0}^{n} \binom{n}{i} \psi^i (1-\psi)^{n-i} \omega^{(n-i)i} \le 0
\end{aligned}
\end{equation}
The difference $D_n$ can be factored as:
\begin{equation}
\label{factor}
D_{{n}} = \begin{cases} \Delta(\psi-1)(2\psi-1)(\omega-1) & \mbox{if } {n} \mbox{ is even} \\  \Delta(\psi-1)(2\psi-1)(\omega-1)(\omega+1) & \mbox{if } {n} \mbox{ is odd} \end{cases}
\end{equation}
where $\Delta$ is a positive polynomial (only numeric proofs are possible and they are given in Tables \ref{tableomegapsi4}-\ref{tableomegapsi12} and in Figure \ref{delta}) and $\mathcal{D}$ denotes the set of positive odd integers. \\
By expressions \ref{tau}-\ref{factor}, it follows that:
\begin{equation*}
\tau_1 \le 1 \iff D_n   \le 0 \iff \begin{cases} 0 \le \psi \le \frac{1}{2}  \wedge 0 \le \omega \le 1 \\ \frac{1}{2} \le \psi < 1  \wedge \omega \ge 1. \end{cases}
\end{equation*}
This result is also shown in Figure \ref{tau1}, where red and black points correspond respectively to $\tau_1 \le 1$ and $\tau_1 >1$, for different sample size $n$.
\end{proof}
\begin{figure}[!htbp]
\centering
\includegraphics[width=\textwidth]{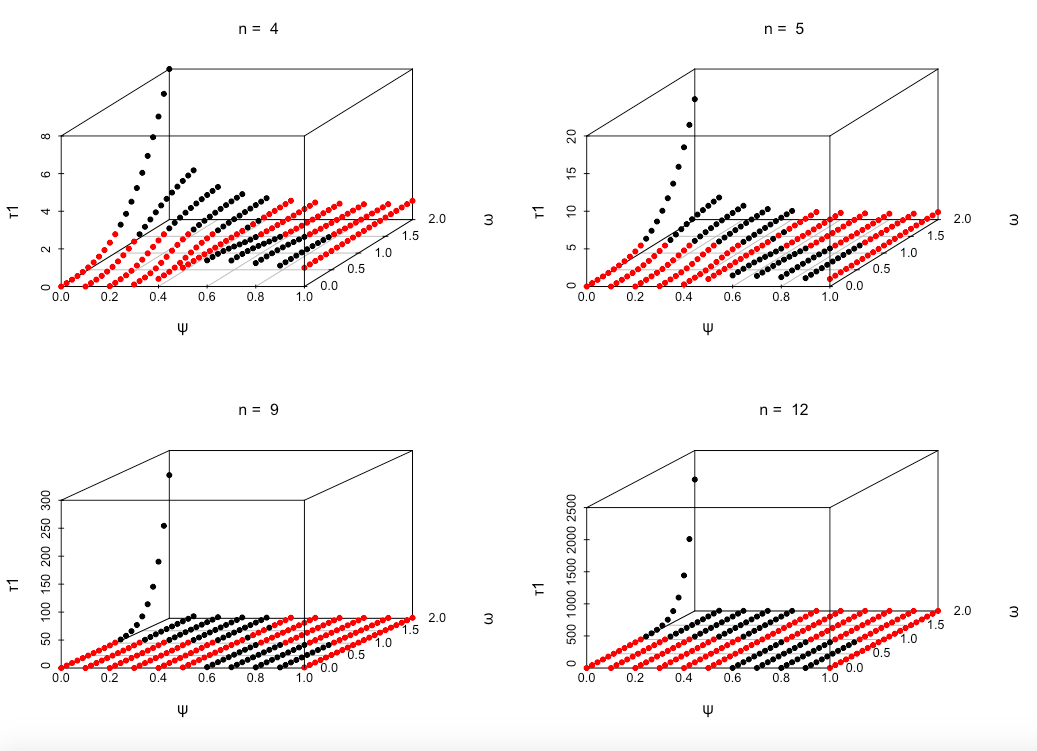}
\caption{Values of $\tau_1$ for $\psi \in [0,1]$, $\omega \in [0,2]$. Red and black points correspond respectively to $\tau_1 \le 1$ and $\tau_1 >1$.}
\label{tau1}
\end{figure}

\begin{figure}[!htbp]
\centering
\includegraphics[width=\textwidth]{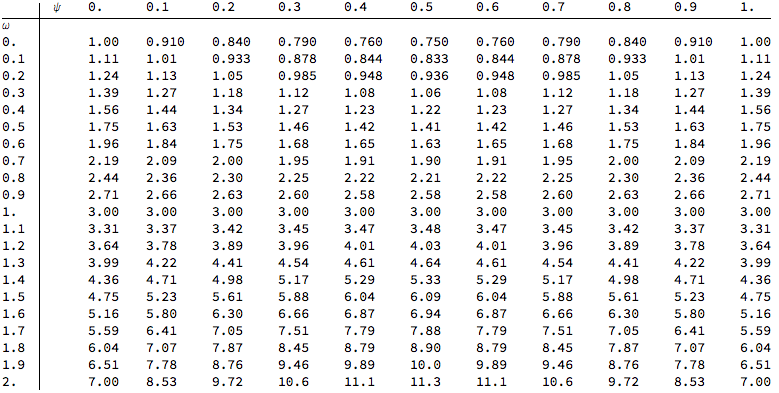}
\captionof{table}[foo]{Values of $\Delta$ for $\psi \in [0,1]$, $\omega \in [0,2]$ and $n=4$.}
\label{tableomegapsi4}
\end{figure}

\begin{figure}[!htbp]
\centering
\includegraphics[width=\textwidth]{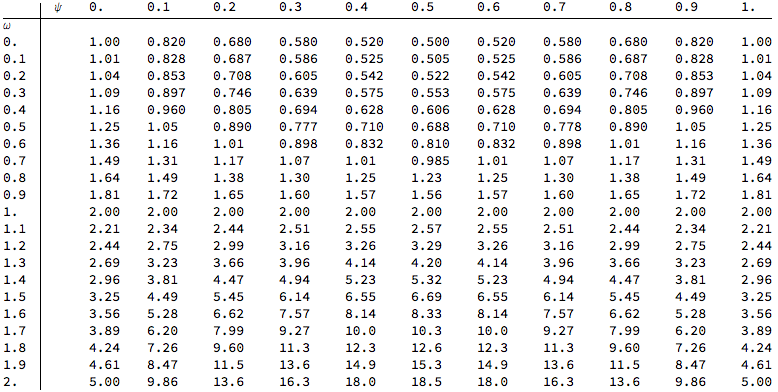}
\captionof{table}[foo]{Values of $\Delta$ for $\psi \in [0,1]$, $\omega \in [0,2]$ and $n=5$.}
\label{tableomegapsi5}
\end{figure}

\begin{figure}[!htbp]
\centering
\includegraphics[width=\textwidth]{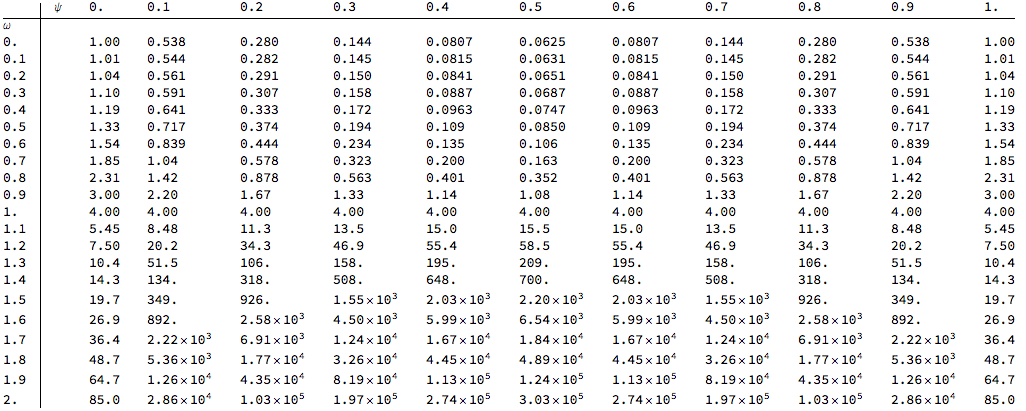}
\captionof{table}[foo]{Values of $\Delta$ for $\psi \in [0,1]$, $\omega \in [0,2]$ and $n=9$.}
\label{tableomegapsi9}
\end{figure}

\begin{figure}[!htbp]
\centering
\includegraphics[width=\textwidth]{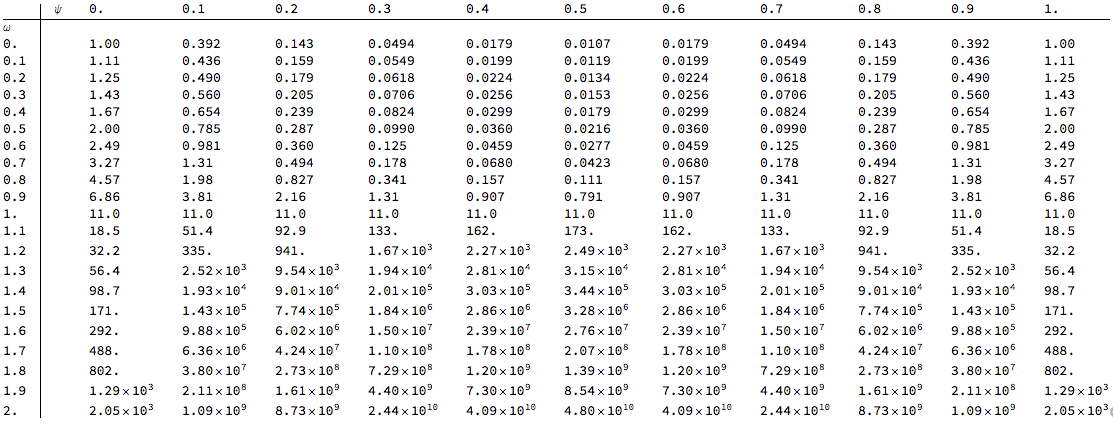}
\captionof{table}[foo]{Values of $\Delta$ for $\psi \in [0,1]$, $\omega \in [0,2]$ and $n=12$.}
\label{tableomegapsi12}
\end{figure}

\begin{figure}[!htbp]
\centering
\includegraphics[width=0.8\textwidth,height=0.8\textheight,keepaspectratio]{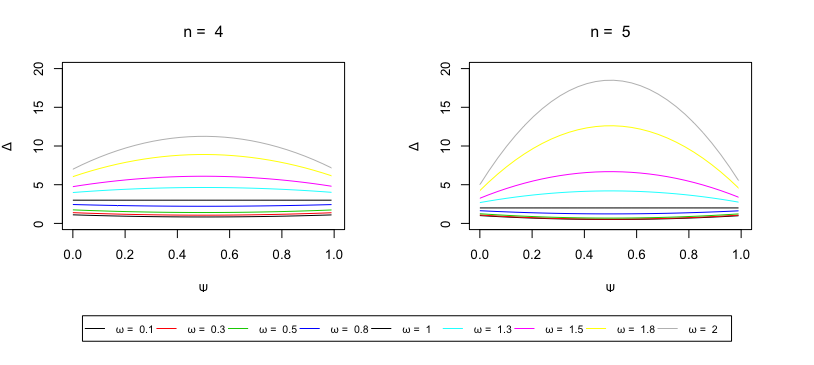}
\includegraphics[width=0.8\textwidth,height=0.8\textheight,keepaspectratio]{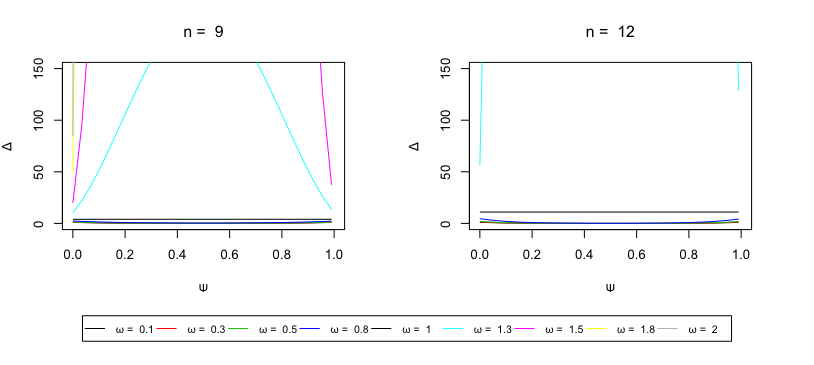}
\caption{Values of $\Delta$ for $\psi \in [0,1]$ according to different values of $\omega$ and $n$.}
\label{delta}
\end{figure}

\section{Concluding remarks}
Our interest in the (L)MBD has been motivated by ensemble classification accuracy assessment: in particular, our idea is to use the (L)MBD in order to describe the classification accuracy of an ensemble of $n$ classifiers. Namely, following such distribution, we can model the prediction accuracy of the majority vote ensemble in the presence of dependent classifiers (with the same individual probability of success $\pi$) as:
\begin{equation*}
\label{eq:A.multbinom}
\widehat{Ac}=1-F_s(q)=\sum_{y=q+1}^{n} \frac{\binom{n}{y} \psi^y (1-\psi)^{n-y} \omega^{(n-y)y}}{\sum_{i=0}^{n} \binom{n}{i} \psi^i (1-\psi)^{n-i} \omega^{(n-i)i}},
\end{equation*}
where $q = \begin{cases} \frac{n}{2} & \mbox{if $n$ is even} \\  \frac{n-1}{2} & \mbox{if $n$ is odd}. \end{cases}$.\\ \\
Specifically, in the RP ensemble classifier context \citep{samworth}, we proved that (L)MBD is able to almost perfectly seize the actual intra-classifiers association using the $\omega$ parameter and, thus, to better characterize and predict (with respect to both the Binomial and Beta-Binomial models) the ensemble accuracy. \\
Moreover, as a consequence of Theorem~\ref{theorem2}, we have proved that the marginal probability of success $\psi$ of a set of $n$ classifiers is larger than the common individual one, $\pi$, only if the $n$ classifiers are negatively related ($\omega \ge 1$) to each other.

\section*{Acknowledgements}
I would like to thank Professor Patricia Altham for her valuable comments and suggestions on a preliminary draft of this paper. All errors remain mine.

\break

\bibliographystyle{chicago}
\bibliography{bibliografiaLMBD}

\begin{thebibliography}{}

\bibitem[\protect\citeauthoryear{Altham}{Altham}{1978}]{altham1978two}
Altham, P. M.~E. (1978).
\newblock Two generalizations of the binomial distribution.
\newblock {\em Journal of the Royal Statistical Society. Series C (Applied
  Statistics)\/}~{\em 27\/}(2), 162--167.

\bibitem[\protect\citeauthoryear{Cannings and Samworth}{Cannings and
  Samworth}{2017}]{samworth}
Cannings, T.~I. and R.~J. Samworth (2017).
\newblock Random projection ensemble classification.
\newblock {\em Royal Statistical Society ser. B\/}.

\bibitem[\protect\citeauthoryear{Diniz, Tutia, and Leite}{Diniz
  et~al.}{2010}]{diniz2010bayesian}
Diniz, C.~A., M.~H. Tutia, and J.~G. Leite (2010).
\newblock Bayesian analysis of a correlated binomial model.
\newblock {\em Brazilian Journal of Probability and Statistics\/}, 68--77.

\bibitem[\protect\citeauthoryear{Kadane et~al.}{Kadane
  et~al.}{2016}]{kadane2016sums}
Kadane, J.~B. et~al. (2016).
\newblock Sums of possibly associated bernoulli variables: The
  conway--maxwell-binomial distribution.
\newblock {\em Bayesian Analysis\/}~{\em 11\/}(2), 403--420.

\bibitem[\protect\citeauthoryear{Kaminski}{Kaminski}{2007}]{kaminski2007central}
Kaminski, M. (2007).
\newblock Central limit theorem for certain classes of dependent random
  variables.
\newblock {\em Theory of Probability \& Its Applications\/}~{\em 51\/}(2),
  335--342.

\bibitem[\protect\citeauthoryear{Lovison}{Lovison}{1998}]{lovison1998alternative}
Lovison, G. (1998).
\newblock An alternative representation of altham's multiplicative-binomial
  distribution.
\newblock {\em Statistics \& Probability Letters\/}~{\em 36\/}(4), 415--420.

\bibitem[\protect\citeauthoryear{Luce{\~n}o}{Luce{\~n}o}{1995}]{luceno1995family}
Luce{\~n}o, A. (1995).
\newblock A family of partially correlated poisson models for overdispersion.
\newblock {\em Computational statistics \& data analysis\/}~{\em 20\/}(5),
  511--520.

\bibitem[\protect\citeauthoryear{Skellam}{Skellam}{1948}]{skellam1948probability}
Skellam, J. (1948).
\newblock A probability distribution derived from the binomial distribution by
  regarding the probability of success as variable between the sets of trials.
\newblock {\em Journal of the Royal Statistical Society. Series B
  (Methodological)\/}~{\em 10\/}(2), 257--261.

\bibitem[\protect\citeauthoryear{Zhao and Prentice}{Zhao and
  Prentice}{1990}]{zhao1990correlated}
Zhao, L.~P. and R.~L. Prentice (1990).
\newblock Correlated binary regression using a quadratic exponential model.
\newblock {\em Biometrika\/}~{\em 77\/}(3), 642--648.

\end{thebibliography}
\end{document}